\documentclass{article}

\usepackage{lineno,hyperref}
\usepackage{amsmath}
\usepackage{amssymb}
\usepackage{amsthm}
\modulolinenumbers[5]
\usepackage{enumitem}
\usepackage[all]{xy}
\SelectTips{cm}{12}

\newtheorem{lemma}{Lemma}
\newtheorem*{lemma*}{Lemma}

\newtheorem{theorem}{Theorem}
\newtheorem*{theorem*}{Theorem}

\newcommand{\dotminus}{\mathbin{\dot{-}}}

\DeclareMathOperator\mat{M}
\DeclareMathOperator\unit{U}

\DeclareMathOperator\stmap{st}

\DeclareMathOperator\glin{GL}

\DeclareMathOperator\stlin{St}

\DeclareMathOperator\Ker{Ker}

\DeclareMathOperator\can{can}
\newcommand{\eps}{\varepsilon}
\newcommand{\leqt}{\trianglelefteq}
\newcommand{\inv}[1]{\!\;\overline{\!\!\:#1\vphantom !\!\!\:}\;\!}

\DeclareMathOperator{\Aut}{Aut}

\DeclareMathOperator{\Spec}{Spec}

\newcommand{\up}[2]{{^{#1}\!{#2}}}

\newcommand{\Set}{\mathbf{Set}}
\newcommand{\Group}{\mathbf{Grp}}
\DeclareMathOperator{\Pro}{Pro}

\title{Cosheaves of Steinberg pro-groups}

\author{
  Egor Voronetsky\thanks{The work was supported by the Theoretical Physics and Mathematics Advancement Foundation ``BASIS'' and by ``Native towns'', a social investment program of PJSC ``Gazprom Neft''.} \\
  Chebyshev Laboratory, \\
  St. Petersburg State University, \\
  14th Line V.O., 29B, \\
  Saint Petersburg 199178 Russia \\
}

\begin{document}
\maketitle

\begin{abstract}
Steinberg pro-groups are certain pro-groups used to analyze ordinary Steinberg groups locally in Zariski topology. In this paper we show that Steinberg pro-groups associated with general linear groups, odd unitary groups, and Chevalley groups satisfy a Zariski cosheaf property as crossed pro-modules over the base groups. Also, we prove an analogue of the standard commutator formulae for relative Steinberg groups. As an application, we show that the base groups over localized rings naturally act on the corresponding Steinberg pro-groups.
\end{abstract}

\section{Introduction}

In this paper \(G(\Phi, A)\) denotes either a Chevalley group of type \(\Phi\) over a unital commutative ring \(A\), or a general linear group constructed by a unital associative ring \(A\) with a complete family of full orthogonal idempotents (and \(\Phi\) is of type \(\mathsf A_\ell\)), or an odd unitary group constructed by an odd form ring \(A = (R, \Delta)\) with a strong orthogonal hyperbolic family in the sense of \cite[\S4]{central-ku2}, \cite[\S3.3]{thesis}, and \cite[\S2]{rel-stu} (and \(\Phi\) is of type \(\mathsf{BC}_\ell\)).

For any such group \(G(\Phi, A)\) there is a corresponding Steinberg group \(\stlin(\Phi, A)\), it is generated by the root elements of \(G(\Phi, A)\) but only with the ``obvious'' relations between them. If the rank of \(\Phi\) is at least \(3\) and \(A\) is locally finite over the base unital commutative ring \(K\), then the canonical homomorphism \(\stlin(\Phi, A) \to G(\Phi, A)\) has central kernel. In some cases (notably, for Chevalley groups of types \(\mathsf{A, C, D, E}\)) this is proved in \cite{central-ko2, central-k2c, central-k2d, central-k2e, central-k2-local, central-k2-tulen, central-k2-vdk}. The general cases are done in \cite{central-k2f, central-ku2, central-k2} using Steinberg pro-groups. If the rank of \(\Phi\) is \(2\), then there are counterexamples for Chevalley groups \cite{rank-2-k2}. A relative version is proved in some cases by constructing a suitable van der Kallen's ``another presentation'' \cite{central-ko2, rel-st-symp, central-k2e}.

The approach with pro-groups is based on the following observation. For any \(k \in K\) we may construct the localization \(A_k\) and the corresponding group \(G(\Phi, A_k)\). Such a group is simpler than the original \(G(\Phi, A)\) since the colimit of such groups for \(k\) not in a prime ideal \(\mathfrak p \leq K\) is the group \(G(\Phi, A_{\mathfrak p})\) with a Gauss decomposition. On the other hand, \(G(\Phi, -)\) is a sheaf in the Zariski topology, so \(G(\Phi, A)\) may be reconstructed by various \(G(\Phi, A_k)\) with sufficiently large \(k\). Recall that in the linear and Chevalley cases \(A_k\) is the \(K\)-module module \(\varinjlim(A \xrightarrow k A \xrightarrow k \ldots)\) with a suitable multiplication

The localization does not work for Steinberg groups, so instead we consider the ``colocalization'' \(A^{(\infty, k)}\), it is a certain pro-algebra. In the linear and Chevalley cases \(A^{(\infty, k)}\) as a pro-\(K\)-module is the formal projective limit \(\varprojlim^{\Pro}(\ldots \xrightarrow k A \xrightarrow k A \xrightarrow k A)\) with a suitable multiplication. The limit of corresponding Steinberg pro-groups \(\stlin(\Phi, A^{(\infty, k)})\) with \(k\) not in a prime ideal \(\mathfrak p \leq K\) admits a natural action of \(G(\Phi, A_{\mathfrak p})\) due to the Gauss decomposition of the latter. Finally, the morphisms from \(\stlin(\Phi, A^{(\infty, k)})\) to the original \(\stlin(\Phi, A)\) allow to construct an action of \(G(\Phi, A)\) on \(\stlin(\Phi, A)\). The action turns out to be a crossed module structure, so in particular the kernel of \(\stlin(\Phi, A) \to G(\Phi, A)\) is central.

In this paper we prove the following results:
\begin{itemize}
\item Steinberg pro-groups \(\stlin(\Phi, A^{(\infty, k)})\), originally defined as formal projective limit of unrelativized Steinberg groups, are actually formal projective limits of relative Steinberg groups. In particular, they are crossed pro-modules over \(\stlin(\Phi, A)\).
\item For any crossed module \(\delta \colon X \to A\) (i.e. an ideal) there is a crossed square
\[\xymatrix@R=30pt@C=60pt@!0{
\stlin(\Phi, A, X) \ar[r] \ar[d] & \stlin(\Phi, A) \ar[d] \\
G(\Phi, X) \ar[r] & G(\Phi, A).
}\]
involving the relative Steinberg group. This is a generalization of the well-known standard commutator formulae for relative elementary groups \cite[\S8]{yoga}
\[[\mathrm E(\Phi, A, I), G(\Phi, A)] = [\mathrm E(\Phi, A), G(\Phi, I)] = \mathrm E(\Phi, A, I).\]
\item Steinberg pro-groups \(\stlin(\Phi, A^{(\infty, k)})\) form a sheaf of crossed pro-modules over \(\stlin(\Phi, A)\) and \(G(\Phi, A)\) in Zariski topology on the site of principal open subsets of \(\Spec(K)\). This implies that Steinberg pro-groups may be defined for all quasi-compact open subsets of \(\Spec(K)\) by the cosheaf property. Moreover, these pro-groups admit natural presentations as pro-groups in terms of Steinberg pro-groups corresponding to the elements of coverings of \(\Spec(K_k)\).
\item For any multiplicative subset \(S \subseteq K\) there is a natural action of \(G(\Phi, S^{-1} A)\) on \(\stlin(\Phi, A^{(\infty, S)})\). The case \(S = K \setminus \mathfrak p\) for a prime ideal \(\mathfrak p\) is used in the localization proofs of centrality of \(\mathrm K_2\), it is a generalization of the main ingredient from the localization-and-patching proof of the normality of the elementary subgroup \cite[\S4]{yoga}. We prove the general version using a variant of the cosheaf property.
\end{itemize}

\section{Groups with commutator relations}

We use the notation \([g, h] = ghg^{-1} h^{-1}\) and \(\up gh = ghg^{-1}\) for group operations. An action of a group \(G\) on a group \(H\) is also usually denoted by \((g, h) \mapsto \up gh\). By a \textit{root system} we mean a finite crystallographic irreducible root system, possibly non-reduced (i.e. of the type \(\mathsf{BC}_\ell\)). Recall that a subset \(\Sigma \subseteq \Phi\) of a root system is called \textit{special closed} if \(\Sigma\) is contained in an open half-space and \((\Sigma + \Sigma) \cap \Phi \subseteq \Sigma\). A \textit{saturated root subsystem} of \(\Phi\) is an intersection of \(\Phi\) with a subspace. If \(X \subseteq \Phi\) is contained in an open half-space, then \(\langle X \rangle\) denotes the smallest special closed subset of \(\Phi\) containing \(X\). The \textit{dimension} of a subset of \(\Phi\) is the dimension of its span.

Let us consider a group \(G\) with a family of homomorphisms \(t_\alpha \colon P_\alpha \to G\) from nilpotent groups \(P_\alpha\), where \(\alpha\) runs over a root system \(\Phi\) of rank \(\ell \geq 3\). Suppose that
\begin{itemize}
\item for any special closed subset \(\Sigma \subseteq \Phi\) the product map \(\prod_{\alpha \in \Sigma \setminus 2 \Sigma} P_\alpha \to G, (p_\alpha)_\alpha \mapsto \prod_\alpha t_\alpha(p_\alpha)\) is injective;
\item if \(\alpha, 2 \alpha \in \Phi\) (i.e. \(\Phi\) is of type \(\mathsf{BC}_\ell\) and \(\alpha\) is ultrashort), then \(P_{2 \alpha} \leq P_\alpha\) and \(t_{2\alpha}(p) = t_\alpha(p)\) for \(p \in P_{2 \alpha}\);
\item if \(\alpha\) and \(\beta\) are not anti-parallel, then the \textit{Chevalley commutator formula} \([t_\alpha(p), t_\beta(q)] = \prod_{\substack{i \alpha + j \beta \in \Phi \\ i, j > 0}} t_{i \alpha + j \beta}(f_{\alpha \beta i j}(p, q))\) holds, where \(f_{\alpha \beta i j} \colon P_\alpha \times P_\beta \to P_{i \alpha + j \beta}\) are some fixed maps.
\end{itemize}
In this case we say that \(G\) is a \textit{group with commutator relations of type} \(\Phi\), this is a modified definition form \cite[\S3]{st-jordan}. The group operations of \(P_\alpha\) are denoted by \(\dotplus\). Clearly, \(P_\alpha\) is abelian unless \(\alpha\) is an ultrashort root and \(\Phi\) is of type \(\mathsf{BC}_\ell\). In the exceptional case \(P_\alpha\) is two-step nilpotent with \([P_\alpha, P_\alpha]^\cdot \leq P_{2 \alpha}\) and \([P_\alpha, P_{2 \alpha}]^\cdot = \dot 0\).

Take a unital commutative ring \(K\). In this paper we study the following three classes of groups \(G(\Phi, A)\) with commutator relations of type \(\Phi\) constructed by various algebraic objects \(A\):
\begin{itemize}
\item Let \(A\) be an integral unital commutative \(K\)-algebra and \(\Phi\) be a reduced root system. Then \(G(\Phi, A) = G^{\mathrm{sc}}(\Phi, A)\) denotes the group of \(A\)-points of the simply connected Chevalley group scheme with the root system \(\Phi\). Here \(P_\alpha(A) = A\) for all \(\alpha \in \Phi\).
\item Let \(\Phi\) be of type \(\mathsf A_\ell\) and \(A\) be a locally finite generalized matrix algebra over \(K\), i.e. a unital associative \(K\)-algebra with a complete family of orthogonal idempotents \(e_1, \ldots, e_{\ell + 1}\) such that \(e_i\) are complete in the sense \(A = A e_i A\) and finitely generated subalgebra of \(A\) is finite. We associate with \(A\) the group \(G(\mathsf A_\ell, A) = \glin(A) = A^*\) with the root homomorphisms
\[t_{\mathrm e_i - \mathrm e_j} = t_{ij} \colon P_{\mathrm e_i - \mathrm e_j}(A) = e_i A e_j \to G(\mathsf A_\ell, A),\, p \mapsto 1 + p\]
for \(i \neq j\). In the case of matrix algebra \(A = \mat(n, K)\) we obtain the group \(G(\mathsf A_{\ell - 1}, K)\) from the first class.
\item Let \(\Phi\) be of type \(\mathsf{BC}_\ell\) and \(A = (R, \Delta)\) be a locally finite \textit{odd form \(K\)-algebra} with a \textit{strong orthogonal hyperbolic family} of rank \(\ell\) in the sense of \cite[\S4]{central-ku2} and \cite[\S3.3]{thesis}. Recall that an odd form \(K\)-algebra \((R, \Delta)\) consists of a non-unital associative \(K\)-algebra \(R\), an involution \(a \mapsto \inv a\) on \(R\), a group \(\Delta\) with group operation \(\dotplus\), group homomorphisms \(\phi \colon R \to \Delta\) and \(\pi \colon \Delta \to R\), a map \(\rho \colon \Delta \to R\), and a right action \((-) \cdot (=)\) of the multiplicative monoid of \(R \rtimes K\) on \(\Delta\) by group endomorphisms such that
\begin{align*}
\pi(\phi(a)) &= 0, &
\phi(a + \inv a) &= \phi(\inv aa k) = \dot 0,
\\
\rho(\phi(a)) &= a - \inv a, &
v \dotplus u &= u \dotplus \phi(\inv{\pi(u)} \pi(v)) \dotplus v,
\\
\phi(a) \cdot b &= \phi(\inv{\,b\,} ab), &
\rho(u \dotplus v) &= \rho(u) - \inv{\pi(u)} \pi(v) + \rho(v),
\\
\pi(u \cdot b) &= \pi(u) b, &
0 &= \rho(u) + \inv{\pi(u)} \pi(u) + \inv{\rho(u)},
\\
\rho(u \cdot b) &= \inv{\,b\,} \rho(u) b, &
u \cdot (b + b') &= u \cdot b \dotplus \phi(\inv{b'} \rho(u) b) \dotplus u \cdot b'.
\end{align*}
for \(u, v \in \Delta\), \(a \in R\), \(b, b' \in R \rtimes K\), \(k \in K\) (the involution on \(R \rtimes K\) is given by \(\inv{a \rtimes k} = \inv a \rtimes k\)). It is called locally finite if each finitely generated \(K\)-subalgebra of \(R\) is finite. A strong hyperbolic orthogonal family consists of idempotents \(e_i \in R\) and elements \(q_i \in \Delta\) for \(-\ell \leq i \leq -1\) and \(1 \leq i \leq \ell\) such that
\begin{align*}
e_i e_j &= 0, &
\inv{e_i} &= e_{-i}, &
e_i &\in R e_j R, \\
\pi(q_i) &= e_i, &
\rho(q_i) &= 0, &
q_i &= q_i \cdot e_i
\end{align*}
for \(i \neq j\). There is a natural \textit{unitary group} \(G(\Phi, A) = \unit(R, \Delta)\) with the homomorphisms
\begin{align*}
t_{\mathrm e_j - \mathrm e_i} = T_{ij} &\colon P_{\mathrm e_j - \mathrm e_i}(A) = e_i R e_j \to G(\Phi, A) \text{ for } i + j > 0 \text{ and } i \neq j; \\
t_{\mathrm e_j} = T_j &\colon P_{\mathrm e_j}(A) = \{u \in \Delta \cdot e_j \mid e_k \pi(u) = 0 \text{ for all } k\} \to G \text{ for } j \neq 0;
\end{align*}
and the restrictions \(t_{2 \mathrm e_j}\) of \(t_{\mathrm e_j}\) on \(P_{2 \mathrm e_j}(A) = \phi(e_{-j} R e_j) \leq P_{\mathrm e_j}(A)\).
\end{itemize}
The unitary groups in the third class generalize, in particular, orthogonal and symplectic groups, see \cite[\S2]{central-ku2} or \cite[\S1.1]{thesis}. In all cases the maps \(f_{\alpha \beta i j}\) may be expressed in terms of the operations on \(A\).

For such a group \(G(\Phi, A)\) we may define a corresponding \textit{Steinberg group} \(\stlin(\Phi, A)\) as the group generated by the elements \(x_\alpha(p)\) for \(\alpha \in \Phi\), \(p \in P_\alpha(A)\) satisfying the \textit{Steinberg relations}
\begin{itemize}
\item \(x_\alpha(p)\, x_\alpha(q) = x_\alpha(p \dotplus q)\);
\item if \(\alpha, 2 \alpha \in \Phi\) and \(p \in P_{2 \alpha}(A)\), then \(x_{2 \alpha}(p) = x_\alpha(p)\);
\item if \(\alpha\) and \(\beta\) are not anti-parallel, then \([x_\alpha(p), x_\beta(q)] = \prod_{\substack{i \alpha + j \beta \in \Phi \\ i, j > 0}} x_{i \alpha + j \beta}(f_{\alpha \beta i j}(p, q))\).
\end{itemize}
Clearly, \(\stlin(\Phi, A)\) together with the maps \(x_\alpha\) is also a group with commutator relations of type \(\Phi\). There is a canonical homomorphism \(\stmap \colon \stlin(\Phi, A) \to G(\Phi, A),\, x_\alpha(p) \mapsto t_\alpha(p)\). For any special closed subset \(\Sigma \subseteq \Phi\) the subgroup \(\stlin(\Sigma, A) = \langle x_\alpha(P_\alpha(A)) \rangle \leq \stlin(\Phi, A)\) injectively maps to \(G(\Phi, A)\) by definition.

We also need relative Steinberg groups. The natural setting to consider relative objects are \textit{algebraically coherent semi-abelian categories} \cite{alg-coh}. In every semi-abelian category \(\mathcal A\) there are notions of internal actions and crossed modules. \textit{Actions} of an object \(A\) of \(\mathcal A\) on an object \(X\) are the isomorphism classes of right split exact sequences \(\dot 0 \to X \to X \rtimes A \leftrightarrows A \to \dot 0\), the middle term is called the \textit{semi-direct product} (it is determined by an action up to the unique isomorphism). A \textit{precrossed module} in \(\mathcal A\) is a morphism \(\delta \colon X \to A\) together with an action of \(A\) on \(X\) such that \(\delta\) is equivariant with respect to this action (\(A\) acts on itself in a canonical way). A \textit{crossed module} is a precrossed module such that the canonical action of \(X\) on itself is the pullback of the action of \(A\) on \(X\).

A \textit{morphism} \((f, g)\) from a crossed module \(\delta \colon X \to A\) to a crossed module \(\delta \colon Y \to B\) consists of morphisms \(f \colon X \to Y\) and \(g \colon A \to B\) such that \(\delta \circ f = g \circ \delta\) and \(f\) is \(A\)-equivariant, where the action of \(A\) on \(Y\) is induced via \(g\).

There is a one-to-one correspondence between isomorphism classes of precrossed modules and \textit{reflexive graphs} in \(\mathcal A\), i.e. the tuples \((A, B, p_1, p_2, d)\), where \(p_i \colon B \to A\) are morphisms with a common section \(d\). Namely, \(X\) corresponds to \(\Ker(p_2)\) and \(\delta\) is induced by \(p_1\). The isomorphism classes of crossed modules correspond to \textit{internal categories} (or \textit{internal groupoids}), i.e. the reflexive graphs with composition morphisms \(\lim(B \xrightarrow{p_2} A \xleftarrow{p_1} B) \to B\) making \((\mathcal A(T, A), \mathcal A(T, B))\) categories for any object \(T\). Such a composition morphism is necessarily unique.

For example, in the category of groups a homomorphism \(\delta \colon X \to G\) is a crossed module if \(G\) acts on \(X\), \(\delta(\up gx) = \up g{\delta(x)}\), and the \textit{Peiffer identity} \(\up xy = \up{\delta(x)}y\) holds for all \(x, y \in X\).

From now on \(\mathcal A\) denotes one of the following algebraically coherent semi-abelian categories:
\begin{itemize}
\item In the Chevalley case \(\mathcal A\) is the category of all commutative \(K\)-algebras, not necessarily unital. In this category an action of \(A\) on \(X\) is a biadditive multiplication \(A \times X \to X\) making \(X\) an \(A\)-algebra. A homomorphism \(\delta \colon X \to A\) is a crossed module if \(A\) acts on \(X\), \(\delta(ax) = a \delta(x)\), and \(xy = \delta(x) y\) for all \(x, y \in X\).
\item In the linear case \(\mathcal A\) is the category of all associative \(K\)-algebras. Here an action of \(A\) on \(X\) is a pair of biadditive multiplications \(A \times X \to X\) and \(X \times A \to X\) making \(X\) an \(A\)-\(A\)-bimodule such that \((ax)y = a(xy)\), \((xa)y = x(ay)\), and \((xy)a = x(ya)\) for all \(x, y \in X\) and \(a \in A\). A homomorphism \(\delta \colon X \to A\) is a crossed module if \(A\) acts on \(X\), \(\delta(ax) = a \delta(x)\), \(\delta(xa) = \delta(x) a\), and \(xy = \delta(x) y = x \delta(y)\) for all \(x, y \in X\).
\item In the unitary case \(\mathcal A\) is the category of odd form \(K\)-algebras. Actions and crossed modules in this category are described in \cite[\S2]{central-ku2}, \cite[\S1.4]{thesis}, and \cite[\S2]{rel-stu}.
\end{itemize}

Objects \(A\) of \(\mathcal A\) from the definition of \(G(\Phi, A)\) above are called \textit{absolute} (i.e. if \(A\) is locally finite and has the identity, a family of idempotents, or a strong orthogonal hyperbolic family). The groups \(P_\alpha(-)\), \(G(\Phi, A)\), and \(\stlin(\Phi, A)\) are functors defined on the subcategory of absolute objects of \(\mathcal A\), where morphisms are the homomorphisms preserving the identity, the idempotents, or the strong orthogonal hyperbolic family respectively. All these functors commute with direct limits, \(P_\alpha(-)\) and \(G(\Phi, -)\) commutes with all finite limits, \(P_\alpha(-)\) and \(\stlin(\Phi, -)\) commutes with surjective homomorphisms.

Let \(\delta \colon X \to A\) be a crossed module over an absolute object in \(\mathcal A\), then \(X \rtimes A\) is also absolute. Since \(G(\Phi, -)\) preserves finite limits, we define \(G(\Phi, X)\) as the kernel of \(G(\Phi, X \rtimes A) \to G(\Phi, A)\), so it is naturally a crossed module over \(G(\Phi, A)\). Clearly, \(G(\Phi, X)\) has commutator relations of type \(\Phi\), the root homomorphisms are \(t_\alpha \colon P_\alpha(X) \to G(\Phi, X)\), and \(P_\alpha(X)\) are the kernels of \(P_\alpha(X \rtimes A) \to P_\alpha(A)\). We define the \textit{unrelativized Steinberg group} \(\stlin(\Phi, X)\) in the same way as \(\stlin(\Phi, A)\), i.e. as the group generated by the elements \(x_\alpha(a)\) for \(\alpha \in \Phi\), \(a \in P_\alpha(X)\) satisfying the Steinberg relations. It has subgroups \(\stlin(\Sigma, X)\) for special closed subsets \(\Sigma \subseteq \Phi\).

Recall that there are natural homomorphisms
\begin{align*}
p_1 \colon X \rtimes A \to A&,\enskip x \rtimes a \mapsto a, \\
p_2 \colon X \rtimes A \to A&,\enskip x \rtimes a \mapsto \delta(x) \dotplus a, \\
d \colon A \to X \rtimes A&,\enskip a \mapsto \dot 0 \rtimes a.
\end{align*}
They induce group homomorphisms \(p_{i*} \colon \stlin(\Phi, X \rtimes A) \to \stlin(\Phi, A)\) and \(d_* \colon \stlin(\Phi, A) \to \stlin(\Phi, X \rtimes A)\), so \(\stlin(\Phi, X \rtimes A) = \Ker(p_{1*}) \rtimes \stlin(\Phi, A)\). A \textit{relative Steinberg group} \cite{rel-k2-keune, rel-k2-loday} is
\[\stlin(\Phi, A, X) = \frac{\Ker(p_{1*})}{[\Ker(p_{1*}), \Ker(p_{2*})]}.\]
It is a crossed module over \(\stlin(\Phi, A)\) (this is why we factor by the commutator). There is a natural homomorphism \(\stlin(\Phi, X) \to \stlin(\Phi, A, X)\), so the relative Steinberg group also has commutator relations of type \(\Phi\).

It is possible to give a presentation of \(\stlin(\Phi, A, X)\) as an abstract group as follows. For any root \(\alpha \in \Phi\) a \textit{thick \(\alpha\)-series} is a set \(\Phi \cap (\mathbb R_{> 0} \beta + \mathbb R \alpha)\) for a root \(\beta \in \Phi\) linearly independent with \(\alpha\). Clearly, \(\Phi \setminus \mathbb R \alpha\) is a disjoint union of thick \(\alpha\)-series, all of them are saturated closed subsets. A thick \(\alpha\)-series \(\Sigma\) has dimension \(1\) or \(2\), in the first case \(\Sigma = \langle \beta \rangle\), where \((\alpha, \beta)\) is a base of a saturated root subsystem of type \(\mathsf A_1 \times \mathsf A_1\) (or \(\mathsf A_1 \times \mathsf{BC}_1\) if \(\Phi\) is of type \(\mathsf{BC}_\ell\)). In the two-dimensional case the line \(\mathbb R \alpha\) is uniquely determined by \(\Sigma\). For any thick \(\alpha\)-series \(\Sigma\) the groups \(\stlin(\langle \pm \alpha \rangle, A)\) normalize \(\stlin(\Sigma, X)\).

\begin{lemma} \label{explicit-presentation}
Let \(X\) be a crossed module over \(A\). Then \(\stlin(\Phi, A, X)\) is the abstract group generated by the elements \(z_\alpha(a, p) = \up{x_{-\alpha}(p)}{x_\alpha(a)}\) for \(a \in P_\alpha(X)\), \(p \in P_{-\alpha}(A)\), \(\alpha \in \Phi \setminus 2 \Phi\) and the elements \(z_\Sigma(g, h) = \up hg\) for \(g \in \stlin(\Sigma, X)\), \(h \in \stlin(-\Sigma, A)\), and a two-dimensional thick \(\alpha\)-series \(\Sigma \subseteq \Phi\). The relations are
\begin{itemize}
\item \(z_\alpha(a \dotplus b, p) = z_\alpha(a, p)\, z_\alpha(b, p)\);
\item \(z_\Sigma(gg', h) = z_\Sigma(g, h)\, z_\Sigma(g', h)\);
\item \(z_\Sigma(x_\alpha(a), x_{-\alpha}(p)) = z_\alpha(a, p)\) for \(\alpha \in \Sigma\);
\item \([z_\alpha(a, p), z_\beta(b, q)] = 1\) for a thick \(\alpha\)-series \(\mathbb R_{> 0} \beta\);
\item \(\up{z_\alpha(a, p)}{z_\Sigma(g, h)} = z_\Sigma\bigl(\up{t_{-\alpha}(p)\, t_\alpha(\delta(a))\, t_{-\alpha}(\dotminus p)}g, \up{t_{-\alpha}(p)\, t_\alpha(\delta(a))\, t_{-\alpha}(\dotminus p)}h\bigr)\) for a thick \(\alpha\)-series \(\Sigma\);
\item \(z_{\langle \alpha, \beta \rangle \setminus \mathbb R \alpha}\bigl(\up{t_\alpha(p)}g, \up{t_\alpha(p)}{(h\, x_\beta(q))}\bigr) = z_{\langle \alpha, \beta \rangle \setminus \mathbb R \beta}\bigl(\up{t_\beta(q)}g, x_\alpha(p)\, h\bigr)\) for a base \((\alpha, \beta)\) of a two-dimensional indecomposable saturated root subsystem of \(\Phi\), \(g \in \stlin(\langle \alpha, \beta \rangle \setminus (\mathbb R \alpha \cup \mathbb R \beta), X)\), \(h \in \stlin(\langle \alpha, \beta \rangle \setminus (\mathbb R \alpha \cup \mathbb R \beta), A)\);
\item \(z_\alpha(a, p \dotplus \delta(b)) = \up{z_{-\alpha}(b, \dot 0)}{z_\alpha(a, p)}\).
\end{itemize}
\end{lemma}
\begin{proof}
This is \cite[theorem 1]{rel-st} for general linear groups, \cite[theorem 2]{rel-st} for simply-laced Chevalley groups, \cite[theorem 2]{rel-stu} for odd unitary groups, and \cite[theorem 3]{rel-stu} for doubly-laced Chevalley groups.
\end{proof}

\section{Steinberg pro-groups}

The groups \(G(\Phi, A)\) are actually sheaves in the Zariski topology in the following sense. Take elements \(k_1, \ldots, k_n \in K\) generating the unit ideal (i.e. a Zariski covering of \(\mathrm{Spec}(K)\) by principal open subschemes \(\mathrm{Spec}(K_{k_i})\)). Any absolute object \(A\) in \(\mathcal A\) has natural localizations \(A_{k_i}\), for odd form algebras see \cite[\S3]{central-ku2} or \cite[\S2.2]{thesis} for details. Then \(G(\Phi, A)\) is the limit of the diagram \(G(\Phi, A_{k_i}) \rightrightarrows G(\Phi, A_{k_i k_j})\) in the category of groups. Similarly, all \(P_\alpha(A)\) are also sheaves in the Zariski topology. Unfortunately, the Steinberg groups \(\stlin(\Phi, A)\) are not Zariski sheaves in general.

Recall that any category \(\mathcal C\) admits a \textit{pro-completion} \(\Pro(\mathcal C)\). Objects of \(\Pro(\mathcal C)\) are \textit{formal projective limits} \(\varprojlim^{\Pro}_i X_i\) of \textit{inverse systems} \((X_i)_i\) in \(\mathcal C\), i.e. contravariant functors from small filtered categories such as \(\mathbb N = (0 \to 1 \to \ldots)\) to \(\mathcal C\). Morphisms in \(\Pro(\mathcal C)\) are given by
\[\Pro(\mathcal C)\bigl(\varprojlim\nolimits_i^{\Pro} X_i, \varprojlim\nolimits^{\Pro}_j Y_j\bigr) = \varprojlim\nolimits_j \varinjlim\nolimits_i \mathcal C(X_i, Y_j),\]
where the limit and the colimit in the right hand side are taken in the category of sets. If \(X = \varprojlim^{\Pro}_i X_i\) is a pro-object with the index category \(\mathcal I\) and \(u \colon \mathcal J \to \mathcal I\) is a cofinal functor, then the canonical morphism \(X \to \varprojlim^{\Pro}_j X_{u(j)}\) is an isomorphism in \(\Pro(\mathcal C)\). The category \(\mathcal C\) embeds into \(\Pro(\mathcal C)\) in a natural way, i.e. the functor \(\mathcal C \to \Pro(\mathcal C), X \mapsto X\) is fully faithful.

Since the category of pro-sets \(\Pro(\Set)\) is cartesian, it is useful to denote various composite morphisms in \(\Pro(\Set)\) using first-order terms. For example, \([x, y]\) denotes the commutator morphism \(G \times G \to G\) for any group object \(G\) (or the additive commutator morphism of a ring object depending on the context), and \(f(g(x), h(y)) = f(h(x), g(y))\) for morphisms \(g, h \colon X \to Y\), \(f \colon Y \times Y \to Z\) means that the composite morphisms \(f \circ (g \times h)\) and \(f \circ (h \times g)\) coincide. We write \(x \in X\) to denote that a pro-set \(X\) is the domain of a formal variable \(x\) in a first-order term.

There is a natural ``forgetful'' functor from \(\Pro(\Group)\) to the category of group objects of pro-sets, it is fully faithful: a morphism \(f \colon G \to H\) of pro-groups considered as pro-sets is a morphism of pro-groups if and only if \(f(xy) = f(x)\, f(y)\) for \(x, y \in G\) (i.e. an equality holds between the corresponding morphisms \(G \times G \to H\) of pro-sets). Similarly, the ``forgetful'' functor from \(\Pro(\mathcal A)\) to the category of \(\mathcal A\)-objects in \(\Pro(\Set)\) is fully faithful.

The categories \(\Pro(\mathcal A)\) and \(\Pro(\Group)\) are algebraically coherent semi-abelian by \cite[\S1.9]{pro-semiabelian}. Actions in these categories are the same as ``ordinary'' actions of the corresponding algebraic objects in \(\Pro(\Set)\) by \cite[theorems 1, 5, 6]{pro-actions} and \cite[theorem 4]{thesis}, e.g. the actions of a pro-object \(G\) on a pro-object \(X\) are given by a family of actions of \(\Pro(\Set)(T, G)\) on \(\Pro(\Set)(T, X)\) natural on the pro-set \(T\). An object \(G\) from \(\mathcal A\) may act on an object \(X\) from \(\Pro(\mathcal A)\) in two ways: as a pro-object (inside \(\Pro(\mathcal A)\)) or as a family of actions of \(G\) on \(\Pro(\Set)(T, X)\) natural on \(T\). In the first case we say that the action is \textit{strong}, in the second case the action is called \textit{weak}. Every strong action induces a corresponding weak action. For example, a strong action of a group \(G\) on a pro-group \(X\) is a morphism \(a \colon G \times X \to X\) of pro-sets such that \(a(gh, x) = a(g, a(h, x))\) and \(a(g, xy) = a(g, x)\, a(g, y)\) for \(g, h \in G\) and \(x, y \in X\). On the other hand, a weak action of \(G\) on \(X\) is a group homomorphism \(G \to \Aut(X)\).

Let \(\mathcal K\) be the following category. Its objects are elements of \(K\), morphisms from \(k\) to \(k'\) are elements \(k''\) such that \(k' = k k''\), the identity morphisms are given by \(1\), and the composition coincides with the multiplication. Clearly, the localization of \(\mathcal K\) by the morphisms \(k^t \colon k \to k^{t + 1}\) is equivalent to the opposite poset of principal open subsets of \(\Spec(K)\), the family of these morphisms satisfies a suitable Ore condition. For every \(k \in K\) there is a functor
\begin{align*}
k^{(-)} \colon \mathbb N \to \mathcal K,\, n &\mapsto k^n, \\
(n \leq m) &\mapsto (k^{m - n} \colon k^n \to k^m).
\end{align*}

For any \(s \in K\) we construct the \textit{\(s\)-homotope} \(\delta \colon A^{(s)} \to A\) (as a crossed module in \(\mathcal A\)) as follows:
\begin{itemize}
\item In the Chevalley and linear cases \(A^{(s)} = \{a^{(s)} \mid a \in A\}\) with the operations \(a^{(s)} + b^{(s)} = (a + b)^{(s)}\), \(\delta\bigl(a^{(s)}\bigr) = as\), \(a b^{(s)} = a^{(s)} b = (ab)^{(s)}\), \(a^{(s)} b^{(s)} = (abs)^{(s)}\).
\item In the unitary case the construction is given in \cite[\S3.2]{thesis}. Namely, if \(A = (R, \Delta)\), then \(R^{(s)}\) is the corresponding homotope of \(R\) as a \(K\)-algebra and \(\Delta^{(s)}\) is the abstract group generated by \(\phi(a^{(s)})\) and \(u^{(s)}\) for \(a \in R\), \(u \in \Delta\) with the following relations: \(\phi \colon R^{(s)} \to \Delta^{(s)}\) is a homomorphism with central image, \((u \dotplus v)^{(s)} = u^{(s)} \dotplus v^{(s)}\), \(\phi(r)^{(s)} = \phi\bigl((rs)^{(s)}\bigr)\), \(\phi\bigl((r + \inv r)^{(s)}\bigr) = \phi\bigl((\inv rrk)^{(s)}\bigr) = \dot 0\). The operations are given by
\begin{align*}
\inv{r^{(s)}} &= {\inv r}^{(s)}, &
\pi\bigl(u^{(s)}\bigr) &= \pi(u)^{(s)}, &
\rho\bigl(u^{(s)}\bigr) &= \bigl(\rho(u) s\bigr)^{(s)}, \\
u^{(s)} \cdot r^{(s)} &= (u \cdot rs)^{(s)}, &
\delta\bigl(u^{(s)}\bigr) &= \delta(u) \cdot s, &
u \cdot r^{(s)} &= u^{(s)} \cdot r = (u \cdot r)^{(s)}.
\end{align*}
\end{itemize}

The homotope construction is a contravariant functor from \(\mathcal K\) to the category of crossed modules over \(A\). Namely, if \(s'' \colon s \to s'\) is a morphism in \(\mathcal K\), then the corresponding homomorphism \(A^{(s')} \to A^{(s)}\) is given by \(a^{(s')} \mapsto (as'')^{(s)}\) and \(u^{(s')} \mapsto (u \cdot s'')^{(s)}\).

The \textit{colocalization} of an object \(A\) of \(\mathcal A\) at a multiplicative subset \(S \subseteq K\) is the formal projective limit \(A^{(\infty, S)} = \varprojlim_{s \in \mathcal S} A^{(s)}\) in \(\Pro(\mathcal A)\) (or in the pro-completion of the category of crossed modules over \(A\)), where the objects of \(\mathcal S \subseteq \mathcal K\) are the elements of \(S\) and the morphisms \(s \to s'\) are \(s'' \in S\) with the property \(s' = s s''\). For any \(k \in K\) we also define the colocalization \(A^{(\infty, k)} = \varprojlim_{n \in \mathbb N} A^{(k^n)}\). It is easy to see that
\begin{itemize}
\item If \(S = \{1, k, k^2, \ldots\}\), then the canonical morphism \(A^{(\infty, S)} \to A^{(\infty, k)}\) is an isomorphism even if not all \(k^n\) are distinct. Indeed, it is given by the precomposition with the cofinal functor \(k^{(-)} \colon \mathbb N \to \mathcal S\).
\item For any multiplicative subset \(S \subseteq K\) we have \(A^{(\infty, S)} \cong \varprojlim_{s \in \mathcal S} A^{(\infty, s)}\), where the projective limit is taken in \(\Pro(\mathcal A)\). Here the right hand side may be written as \(\varprojlim_{(s, n) \in \mathcal S \times \mathbb N}^{\Pro} A^{(s^n)}\), it is given by the precomposition with the cofinal functor \((-)^{(=)} \colon \mathcal S \times \mathbb N \to \mathcal S\).
\item For any \(k \in K\) and \(n > 0\) the canonical morphism \(A^{(\infty, k^n)} \to A^{(\infty, k)}\) induced by \(k^{n - 1} \colon k \to k^n\) is an isomorphism. Its inverse is given by the precomposition with the multiplication by \(n\) on the index category \(\mathbb N\).
\end{itemize}

The pro-groups \(G\bigl(\Phi, A^{(\infty, k)}\bigr)\) and \(G\bigl(\Phi, A^{(\infty, S)}\bigr)\) are well-defined as the formal projective limits of the corresponding relative groups, they have the root morphisms \(t_\alpha \colon P_\alpha\bigl(A^{(\infty, k)}\bigr) \to G\bigl(\Phi, A^{(\infty, k)}\bigr)\) and \(t_\alpha \colon P_\alpha\bigl(A^{(\infty, S)}\bigr) \to G\bigl(\Phi, A^{(\infty, S)}\bigr)\) of pro-groups satisfying the definition of group with commutator relations, but with injective homomorphisms replaced by monomorphisms. Recall that monomorphisms in \(\Pro(\Group)\) coincide with formal projective limits of injective group homomorphisms.

The corresponding Steinberg pro-groups \(\stlin\bigl(\Phi, A^{(\infty, k)}\bigr)\) and \(\stlin\bigl(\Phi, A^{(\infty, S)}\bigr)\) may be defined as in \cite[\S2.4]{central-k2f}, \cite[\S5]{central-ku2}, \cite[\S4]{central-k2}, i.e. as the formal projective limits of unrelativized Steinberg groups over homotopes of \(A\). Alternatively, these pro-groups may be given by ``generators and relations'' in the following sense. Suppose that we have pro-sets \(X\), \(Y_1\), \ldots, \(Y_m\) and morphisms of pro-sets \(g_{ij} \colon Y_i \to X\) for \(1 \leq j \leq n_i\). We say that a pro-group \(G\) together with a morphism of pro-sets \(f \colon X \to G\) is \textit{given by generators} \(X\) \textit{and relations} \(\prod_{j = 1}^{n_i} g_{ij}(y)^{\eps_{ij}} = 1\) for \(1 \leq i \leq m\) and some \(\eps_{ij} \in \{-1, 1\}\) if
\begin{itemize}
\item these relations hold in \(G\), i.e. \(\prod_{j = 1}^{n_i} f(g_{ij}(y))^{\eps_{ij}} = 1\) as the corresponding morphisms \(Y_i^{n_i} \to G\) of pro-sets;
\item the pair \((G, f)\) satisfies the universal property in the following sense: for any pro-set \(P\) and a morphism \(f' \colon P \times X \to G'\) of pro-sets satisfying \(\prod_{j = 1}^{n_i} f'(p, g_{ij}(y))^{\eps_{ij}} = 1\) for all \(i\) there is a unique \(h \colon P \times G \to G'\) such that \(h(p, xy) = h(p, x) h(p, y)\) and \(f'(p, x) = h(p, f(x))\).
\end{itemize}
The factor \(P\) in the second condition is needed since \(\Pro(\Set)\) is not closed. The next easy lemma is proved in \cite[\S3.1]{thesis}.

\begin{lemma} \label{universal-pro-group}
Let \(X\), \(Y_1\), \ldots, \(Y_m\) be pro-sets, \(g_{ij} \colon Y_i \to X\) be morphisms of pro-sets for \(1 \leq j \leq n_i\), and \(\eps_{ij} \in \{-1, 1\}\) be integers. Then there exists a pro-group generated by \(X\) with the relations \(\prod_{j = 1}^{n_i} g_{ij}(y)^{\eps_{ij}} = 1\), necessarily unique up to a unique isomorphism. Moreover, if the diagram of \(X\), \(Y_i\), \(g_{ij}\) is the formal projective limit of the corresponding diagrams in \(\Set\), then the required pro-group is the formal projective limit of the groups with the corresponding presentations.
\end{lemma}

It follows that \(\stlin\bigl(\Phi, A^{(\infty, k)}\bigr)\) and \(\stlin\bigl(\Phi, A^{(\infty, S)}\bigr)\) may be defined as the pro-groups generated by \(P_\alpha\bigl(A^{(\infty, k)}\bigr)\) and \(P_\alpha\bigl(A^{(\infty, S)}\bigr)\) (i.e. by their coproduct in \(\Pro(\Set)\)) satisfying the Steinberg relations. By the universal property, there are pro-group morphisms \(\stmap \colon \stlin\bigl(\Phi, A^{(\infty, k)}\bigr) \to G\bigl(\Phi, A^{(\infty, k)}\bigr)\) such that \(\stmap(x_\alpha(p)) = t_\alpha(p)\).

Clearly, \(A^{(\infty, k)}\), \(P_\alpha\bigl(A^{(\infty, k)}\bigr)\), \(G\bigl(\Phi, A^{(\infty, k)}\bigr)\), and \(\stlin\bigl(\Phi, A^{(\infty, k)}\bigr)\) form copresheaves of pro-\(\mathcal A\)-objects or pro-groups on the site of principal open subsets of \(\Spec(K)\) with Zariski topology. If \(\mathfrak p \leqt K\) is a prime ideal, then we write \((\infty, \mathfrak p)\) instead of \((\infty, K \setminus \mathfrak p)\) in the upper indices. The costalks of these copresheaves are \(A^{(\infty, \mathfrak p)}\), \(P_\alpha\bigl(A^{(\infty, \mathfrak p)}\bigr)\), \(G\bigl(\Phi, A^{(\infty, \mathfrak p)}\bigr)\), and \(\stlin\bigl(\Phi, A^{(\infty, \mathfrak p)}\bigr)\) respectively.

\begin{lemma} \label{root-elimination}
Let \(\alpha \in \Phi\) be a root, \(S \subseteq K\) be a multiplicative subset, and \(\stlin\bigl(\Phi \setminus \alpha, A^{(\infty, S)}\bigr)\) be the pro-group with the same presentation as \(\stlin\bigl(\Phi, A^{(\infty, S)}\bigr)\), but without the generators \(x_\beta(p)\) for \(\beta \parallel \alpha\) and the relations
\begin{itemize}
\item \(x_\beta(p \dotplus q) = x_\beta(p)\, x_\beta(q)\) for \(\beta \parallel \alpha\);
\item \(x_{2 \beta}(p) = x_\beta(p)\) for \(\beta \parallel \alpha\);
\item \([x_\beta(p), x_\gamma(q)] = \prod_{\substack{i \beta + j \gamma \in \Phi\\ i, j > 0}} x_{i \beta + j \gamma}(f_{\beta \gamma i j}(p, q))\) for \(\alpha \in \mathbb R_{\geq 0} \beta + \mathbb R_{\geq 0} \gamma\).
\end{itemize}
Then the canonical homomorphism \(\stlin\bigl(\Phi \setminus \alpha, A^{(\infty, S)}\bigr) \to \stlin\bigl(\Phi, A^{(\infty, S)}\bigr)\) is an isomorphism. If \(\alpha, \beta \in \Phi\) are linearly independent, then \(\stlin\bigl(\Phi, A^{(\infty, S)}\bigr)\) as a pro-group is generated by \(x_\gamma\bigl(P_\gamma\bigl(A^{(\infty, S)}\bigr)\bigr)\) for \(\gamma \notin \mathbb R \alpha + \mathbb R \beta\).
\end{lemma}
\begin{proof}
For general linear groups this is \cite[proposition 1]{central-k2}, for odd unitary groups this is proved in \cite[propositions 1, 2]{central-ku2} or \cite[theorem 5]{thesis}. For exceptional Chevalley groups almost the same result is proved in \cite[theorem 1]{central-k2f}, but there only the generators \(x_\alpha(a)\) are omitted. The same proof may be applied for our claim.

Alternatively, the weaker result from \cite{central-k2f} suffices for all application in the current paper if in the Chevalley case we define \(D_\alpha(\Phi, A)\) as the subgroup generated by the maximal torus and \(t_\alpha(P_\alpha(A))\) (instead of the definition given in the next section).
\end{proof}

It is natural to consider another possible definition of Steinberg pro-groups using relativization. Since \(A^{(\infty, k)}\) is a formal projective limit of crossed modules over \(A\) up to an isomorphism, we define \(\stlin'(\Phi, A^{(\infty, k)})\) as the formal projective limit of the corresponding relative Steinberg groups. It has a presentation from lemma \ref{explicit-presentation} in the sense of pro-groups. There is a natural morphism of pro-groups \(\stlin\bigl(\Phi, A^{(\infty, k)}\bigr) \to \stlin'\bigl(\Phi, A^{(\infty, k)}\bigr)\).

\begin{theorem} \label{st-pro-group}
Let \(k \in K\) and \(A\) be an absolute object in \(\mathcal A\). Then the morphism \(\stlin\bigl(\Phi, A^{(\infty, k)}\bigr) \to \stlin'\bigl(\Phi, A^{(\infty, k)}\bigr)\) of pro-groups is an isomorphism.
\end{theorem}
\begin{proof}
Let us construct a strong action of each \(\stlin(\langle \alpha \rangle, A)\) on \(\stlin\bigl(\Phi, A^{(\infty, k)}\bigr)\) by the formula
\[\up{x_\alpha(p)}{g} = F_\alpha\bigl(\up{x_\alpha(p)}{F_\alpha^{-1}(g)}\bigr),\]
where \(F_\alpha \colon \stlin\bigl(\Phi \setminus \alpha, A^{(\infty, k)}\bigr) \to \stlin\bigl(\Phi, A^{(\infty, k)}\bigr)\) is the isomorphism from lemma \ref{root-elimination}. The action on \(\stlin\bigl(\Phi \setminus \alpha, A^{(\infty, k)}\bigr)\) from the right hand side is given by the Chevalley commutator formula. Since all Steinberg relations involve only roots from saturated root subsystems of dimension at most \(2\), the second claim of lemma \ref{root-elimination} implies that these strong actions satisfy the Steinberg relations.

This does not yet gives a strong action of \(\stlin(\Phi, A)\) on \(\stlin\bigl(\Phi, A^{(\infty, k)}\bigr)\). But we may construct the morphisms
\begin{align*}
z_\alpha &\colon P_\alpha\bigl(A^{(\infty, k)}\bigr) \times P_{-\alpha}(A) \to \stlin\bigl(\Phi, A^{(\infty, k)}\bigr), \\
z_\Sigma &\colon \stlin\bigl(\Sigma, A^{(\infty, k)}\bigr) \times \stlin(-\Sigma, A) \to \stlin\bigl(\Phi, A^{(\infty, k)}\bigr)
\end{align*}
of pro-sets from lemma \ref{explicit-presentation} using the strong actions of \(P_\alpha(A)\). It is easy to check that they satisfy the relations, so the morphism \(f\) from the statement has a retraction.

Finally, it remains to check that all generators of \(\stlin'\bigl(\Phi, A^{(\infty, k)}\bigr)\) factor through the \(f\). Indeed,
\[z_\alpha(a, p) = \up{x_{-\alpha}(p)}{f(F_\alpha(F_\alpha^{-1}(x_\alpha(a))))} = f\bigl(F_\alpha\bigl(\up{x_{-\alpha}(p)}{F_\alpha^{-1}(x_\alpha(a))}\bigr)\bigr) = f(z_\alpha(a, p))\]
with values in \(\stlin'\bigl(\Phi, A^{(\infty, k)}\bigr)\). The generators \(z_\Sigma(g, h)\) may be expressed in terms of \(z_\alpha(a, p)\) using the relations from lemma \ref{explicit-presentation}.
\end{proof}

By taking projective limits in \(\Pro(\Group)\) we obtain from theorem \ref{st-pro-group} that \(\stlin\bigl(\Phi, A^{(\infty, S)}\bigr)\) is the formal projective limit of relative Steinberg groups for every multiplicative subset \(S \subseteq K\).

\section{Standard commutator formulae for Steinberg groups}

In this section we prove analogues of the standard commutator formulae. Recall the absolute case:
\begin{lemma} \label{k2-centrality}
The homomorphism \(\stlin(\Phi, A) \to G(\Phi, A)\) is a crossed module in a unique way. The crossed module structure is natural on \(A\).
\end{lemma}
\begin{proof}
Existence and uniqueness is proved in \cite[theorem 3]{central-k2f}, \cite[theorem 3]{central-ku2}, \cite[theorem 2]{central-k2}, and \cite[theorem 7]{thesis}. Now let \(f \colon A \to B\) be a morphism of absolute objects. Since \(\stlin(\Phi, A)\) is perfect, it suffices to check that
\[[\up{f_*(g)}{f_*(x)}, \up{f_*(g)}{f_*(y)}] = [f_*(\up gx), f_*(\up gy)]\]
for every \(x, y \in \stlin(\Phi, A)\) and \(g \in G(\Phi, A)\). But \(\up{f_*(g)}{f_*(x)} \equiv f_*(\up gx)\) modulo \(\mathrm K_2(\Phi, B) = \Ker(\stlin(\Phi, B) \to G(\Phi, B))\) and \(\mathrm K_2(\Phi, B)\) is central in \(\stlin(\Phi, B)\), so the claim follows.
\end{proof}

For all \(\alpha \in \Phi\) there are subgroups \(D_\alpha(\Phi, A) \leq G(\Phi, A)\) normalizing \(\stlin(\Sigma, A) \leq G(\Phi, A)\) for all thick \(\alpha\)-series \(\Sigma\) and containing both the maximal torus (in the appropriate sense) and the root subgroups \(t_{\pm \alpha}(P_{\pm \alpha}(A))\). Namely,
\begin{itemize}
\item In the Chevalley case \(D_\alpha(\Phi, A)\) is generated by the maximal torus and the root subgroups \(t_{\pm \alpha}(P_{\pm \alpha}(A))\) as a group subscheme. As an abstract group it is generated by the maximal torus and a subgroup isomorphic to \(\mathrm{SL}(2, A)\) or \(\mathrm{PGL}(2, A)\).
\item In the linear case \(D_\alpha(\Phi, A)\) is the diagonal subgroup with respect to the new family of idempotents obtained by ``eliminating'' \(\alpha\) \cite[\S2]{central-k2}. Zariski locally it is generated by the ordinary diagonal subgroup and the root subgroups \(t_{\pm \alpha}(P_{\pm \alpha}(A))\).
\item In the unitary case \(D_\alpha(\Phi, A)\) is the diagonal subgroup with respect to the new orthogonal hyperbolic family obtained by ``eliminating'' \(\alpha\), see \cite[\S5]{central-ku2} or \cite[\S1.7]{thesis}.
\end{itemize}
These subgroups are natural on \(A\), they are Zariski subsheaves of \(G(\Phi, A)\) and commute with direct limits and finite limits on \(A\).

For any multiplicative subset \(S \subseteq K\) the object \(S^{-1} A\) weakly acts on \(A^{(\infty, S)}\), so the group \(G(\Phi, S^{-1} A)\) weakly acts on the pro-group \(G\bigl(\Phi, A^{(\infty, S)}\bigr)\). These weak actions are natural on \(A\) and extranatural on \(S\) in the following sense: for any multiplicative subsets \(S \subseteq S' \subseteq K\) the morphism \(A^{(\infty, S')} \to A^{(\infty, S)}\) is \(S^{-1} A\)-equivariant and \(G\bigl(\Phi, A^{(\infty, S')}\bigr) \to G\bigl(\Phi, A^{(\infty, S)}\bigr)\) is \(G(\Phi, S^{-1} A)\)-equivariant.
\begin{lemma} \label{local-actions}
Let \(\mathfrak p \leqt K\) be a prime ideal. Then the group \(G(\Phi, A_{\mathfrak p})\) is generated by \(D_\alpha(\Phi, A_{\mathfrak p})\) for \(\alpha \in \Phi\). There is a unique weak action of \(G(\Phi, A_{\mathfrak p})\) on \(\stlin\bigl(\Phi, A^{(\infty, \mathfrak p)}\bigr)\) making the diagram
\[\xymatrix@R=30pt@C=96pt@!0{
\stlin\bigl(\Sigma, A^{(\infty, \mathfrak p)}\bigr) \ar@{>->}[r] \ar@{>->}[dr] & \stlin
\bigl(\Phi, A^{(\infty, \mathfrak p)}\bigr) \ar[d]^{\stmap} \\
& G\bigl(\Phi, A^{(\infty, \mathfrak p)}\bigr)
}\]
\(D_\alpha(\Phi, A_{\mathfrak p})\)-equivariant for any root \(\alpha\) and thick \(\alpha\)-series \(\Sigma\). The weak action is natural on \(A\). The morphism \(\stlin(\Phi, A^{(\infty, \mathfrak p)}) \to \stlin(\Phi, A)\) is \(G(\Phi, A)\)-equivariant, where \(G(\Phi, A)\) acts on the Steinberg pro-group weakly via \(G(\Phi, A_{\mathfrak p})\).
\end{lemma}
\begin{proof}
This is \cite[proposition 4.3]{central-k2f}, \cite[theorem 1]{central-ku2}, \cite[proposition 3]{central-k2}, and \cite[lemma 26]{thesis}.
\end{proof}

The copresheaves \(k \mapsto A^{(\infty, k)}\) and \(k \mapsto \stlin\bigl(\Phi, A^{(\infty, k)}\bigr)\) are ``coseparated'':
\begin{lemma} \label{costalks}
Let \(S \subseteq K\) be a multiplicative subset. Then the family of morphisms \(A^{(\infty, \mathfrak p)} \to A^{(\infty, S)}\) is jointly epimorphic in \(\Pro(\Group)\), where \(\mathfrak p\) runs over all prime ideals \(\mathfrak p \leqt K\) disjoint with \(S\). The same holds for the families \(P_\alpha\bigl(A^{(\infty, \mathfrak p)}\bigr) \to P_\alpha\bigl(A^{(\infty, S)}\bigr)\) and \(\stlin\bigl(\Phi, A^{(\infty, \mathfrak p)}\bigr) \to \stlin\bigl(\Phi, A^{(\infty, S)}\bigr)\).
\end{lemma}
\begin{proof}
The unitary case is proved in \cite[lemma 27]{thesis}. Otherwise let \(f, g \colon A^{(\infty, S)} \to B\) be morphisms of pro-groups such that their restrictions on all \(A^{(\infty, \mathfrak p)}\) with \(\mathfrak p \cap S = \varnothing\) coincide. Without loss of generality, \(B\) is an abstract group and \(f, g\) are represented by homomorphisms \(A^{(s)} \to B\) for some \(s \in S\). Let
\[\mathfrak a = \bigl\{k \in K \mid f\bigl(ka^{(s)}\bigr) = g\bigl(ka^{(s)}\bigr) \text{ for all } a \in A\bigr\},\]
it is an ideal of \(K\). By assumption, it is not contained in any prime ideal disjoint with \(S\), so \(\mathfrak a\) intersects \(S\). In other words, \(f = g\) as morphisms of pro-groups. The claims for \(P_\alpha\bigl(A^{(\infty, S)}\bigr)\) and \(\stlin\bigl(\Phi, A^{(\infty, S)}\bigr)\) now easily follow.
\end{proof}

Recall that a \textit{crossed square} of groups consists of a commutative diagram
\[\xymatrix@R=30pt@C=48pt@!0{
L \ar[r]^{\widehat \nu} \ar[d]_{\widehat \mu} & N \ar[d]^\nu \\
M \ar[r]^\mu & P
}\]
in the category of groups together with actions of \(P\) on \(M\), \(N\), \(L\) and a map \(h \colon M \times N \to L\) such that
\begin{itemize}
\item \(L\), \(M\), and \(N\) are crossed modules over \(P\);
\item \(\widehat \mu\), \(\widehat \nu\), and \(h\) are \(P\)-equivariant;
\item \(h(mm', n) = \up{\mu(m)}{h(m', n)}\, h(m, n)\), \(h(m, nn') = h(m, n)\, \up{\nu(n)}{h(m, n')}\);
\item \(\widehat \mu(h(m, n)) = m\, \up{\nu(n)}{m^{-1}}\), \(\widehat \nu(h(m, n)) = \up{\mu(m)}n\, n^{-1}\);
\item \(h(m, \widehat \nu(l)) = \up{\mu(m)}l\, l^{-1}\), \(h(\widehat \mu(l), n) = l\, \up{\nu(n)}{l^{-1}}\).
\end{itemize}
The map \(h\) is called a \textit{crossed pairing}. It clearly satisfies \(h(m, 1) = h(1, n) = 1\), \(h(m^{-1}, n) = \up{\mu(m)^{-1}}{h(m, n)^{-1}}\), \(h(m, n^{-1}) = \up{\nu(n)^{-1}}{h(m, n)^{-1}}\).

A commutative diagram
\[\xymatrix@R=24pt@C=42pt@!0{
L \ar[dr]^f \ar[rr]^{\widehat \nu} \ar[dd]_{\widehat \mu} && N \ar[dr]^f \ar[dd]^(0.7)\nu \\
& L' \ar[rr]^(0.3){\widehat \nu} \ar[dd]_(0.3){\widehat \mu} && N' \ar[dd]^\nu \\
M \ar[dr]^f \ar[rr]^(0.7)\mu && P \ar[dr]^f \\
& M' \ar[rr]^\mu && P'
}\]
is a \textit{morphism of crossed squares} if the back and the front faces are crossed squares, \(f(\up pl) = \up{f(p)}{f(l)}\), \(f(\up pn) = \up{f(p)}{f(n)}\), \(f(\up pm) = \up{f(p)}{f(m)}\) (i.e. the diagram induces morphisms of crossed modules), and \(f(h(m, n)) = h(f(m), f(n))\).

We are going to construct a crossed square by the diagram
\[\xymatrix@R=30pt@C=60pt@!0{
\stlin(\Phi, A, X) \ar[r]^\delta \ar[d]_{\stmap} & \stlin(\Phi, A) \ar[d]^\stmap \\
G(\Phi, X) \ar[r]^\delta & G(\Phi, A).
}\tag{*}\label{*}\]
for any crossed module \(\delta \colon X \to A\) over an absolute object in \(\mathcal A\). It is easy to see using lemma \ref{k2-centrality} that there is an action of \(G(\Phi, X \rtimes A) = G(\Phi, X) \rtimes G(\Phi, A)\) on \(\stlin(\Phi, X \rtimes A) = \Ker(p_{1*}) \rtimes \stlin(\Phi, A)\) given by
\[\up{(g \rtimes u)}{(x \rtimes a)} = \bigl(\up g{(\up ux)}\, \langle g, \up u a \rangle\bigr) \rtimes \up u a,\]
for some actions of \(G(\Phi, A)\) and \(G(\Phi, X)\) on \(\Ker(p_{1*})\) and a map \(\langle -, = \rangle \colon G(\Phi, X) \times \stlin(\Phi, A) \to \Ker(p_{1*})\) satisfying
\begin{align*}
\up u{(\up ax)} &= \up{\up u a}{(\up u x)}, &
\delta(\up gx) &= \up{\delta(g)}{\delta(x)}, &
\stmap(\up gx) &= \up g{\stmap(x)},
\\
\up u{(\up gx)} &= \up{\up u g}{(\up u x)}, &
\delta(\up ux) &= \up u{\delta(x)}, &
\stmap(\up ux) &= \up u{\stmap(x)},
\\
\langle g, ab \rangle &= \langle g, a \rangle\, \up a{\langle g, b \rangle}, &
\delta(\langle g, a \rangle) &= \up{\delta(g)}a\, a^{-1}, &
\stmap(\langle g, a \rangle) &= g\, \up{\stmap(a)}{g^{-1}},
\\
\langle gh, a \rangle &= \up g{\langle h, a \rangle}\, \langle g, a \rangle, &
\langle \stmap(x), a \rangle &= x\, \up a{x^{-1}}, &
\up xy &= \up{\stmap(x)}y,
\\
\up u{\langle g, a \rangle} &= \langle \up u g, \up u a \rangle, &
\langle g, a \rangle\, \up a{(\up g x)} &= \up g{(\up a x)}\, \langle g, a \rangle, &
\up ax &= \up{\stmap(a)}x.
\end{align*}
for \(x, y \in \Ker(p_{1*})\), \(a, b \in \stlin(\Phi, A)\), \(g, h \in G(\Phi, X)\), \(u, v \in G(\Phi, A)\).

By lemma \ref{local-actions}, for any prime ideal \(\mathfrak p \leqt K\) the groups \(G(\Phi, X_{\mathfrak p})\) and \(G(\Phi, A_{\mathfrak p})\) weakly act on \(\Ker\bigl(p_{1*}^{(\infty, \mathfrak p)}\bigr)\) and for any \(g \in G(\Phi, X_{\mathfrak p})\) there is a morphism \(\langle g, {-} \rangle \colon \stlin\bigl(\Phi, A^{(\infty, \mathfrak p)}\bigr) \to \Ker\bigl(p_{1*}^{(\infty, \mathfrak p)}\bigr)\) of pro-sets satisfying the same \(15\) identities. Moreover,
\[\up g{\can(x)} = \can(\up gx), \enskip \up u{\can(x)} = \can(\up ux), \enskip \langle g, \can(a) \rangle = \can(\langle g, a \rangle)\]
for any \(g \in G(\Phi, X)\) and \(u \in G(\Phi, A)\), where \(\can \colon \Ker\bigl(p_{1*}^{(\infty, \mathfrak p)}\bigr) \to \Ker(p_{1*})\) and \(\can \colon \stlin\bigl(\Phi, A^{(\infty, \mathfrak p)}\bigr) \to \stlin(\Phi, A)\) are the canonical morphisms.

\begin{theorem} \label{st-x-square}
Let \(\delta \colon X \to A\) be a crossed module over an absolute object in \(\mathcal A\). Then there are a unique action of \(G(\Phi, A)\) on \(\stlin(\Phi, A, X)\) and a unique crossed pairing
\[\langle {-}, {=} \rangle \colon G(\Phi, X) \times \stlin(\Phi, A) \to \stlin(\Phi, A, X)\]
such that the action of \(\stlin(\Phi, A)\) on \(\stlin(\Phi, A, X)\) factors through the action of \(G(\Phi, A)\) and (\ref{*}) is a crossed square. This crossed square is natural on \(\delta \colon X \to A\).
\end{theorem}
\begin{proof}
We show that the crossed square structure is given by the action of \(G(\Phi, A)\) on \(\Ker(p_{1*})\) and the map \(\langle {-}, {=} \rangle \colon G(\Phi, X) \times \stlin(\Phi, A) \to \Ker(p_{1*})\) (so it is clearly natural on \(\delta \colon X \to A\)). Let \(N = [\Ker(p_{1*}), \Ker(p_{2*})] \leq \Ker(p_{1*})\), so \(\stlin(\Phi, A, X) = \Ker(p_{1*}) / N\). It is easy to see that \(N\) is generated by the elements \(\up x y\, \up{\delta(x)}{y^{-1}}\) for \(x, y \in \Ker(p_{1*})\), it is \(G(\Phi, A)\)-invariant.

In order to show that \(N\) is \(G(\Phi, X)\)-invariant we check that
\[\up{\delta(g)}x \equiv \up gx \pmod N\]
for \(g \in G(\Phi, X)\) and \(x \in \Ker(p_{1*})\). Since both sides are homomorphisms on \(x\), by lemma \ref{costalks} it suffices to check that
\[\up{\delta(g)}x \equiv \up gx \pmod{N^{(\infty, \mathfrak p)}},\]
for any prime ideal \(\mathfrak p \leqt K\), \(g \in G(\Phi, X_{\mathfrak p})\), and a formal variable \(x \in \Ker\bigl(p_{1*}^{(\infty, \mathfrak p)}\bigr)\). Here \(N^{(\infty, \mathfrak p)} \leq \Ker\bigl(p_{1*}^{(\infty, \mathfrak p)}\bigr)\) is the sub-pro-group generated by \(\up x y\, \up{\delta(x)}{y^{-1}}\), it is clearly \(\stlin\bigl(\Phi, A^{(\infty, \mathfrak p)}\bigr)\)-invariant (and, in particular, normal). It suffices to check the identity only for generators \(g \in G(\Phi, X_{\mathfrak p})\). Using the identities \(\up{\up ug}x = \up u{\bigl(\up g {\bigl(\up{u^{-1}}x\bigr)} \bigr)}\), \(\up{\up u{\delta(g)}}x = \up u{\bigl(\up{\delta(g)}{\bigl(\up{u^{-1}}x\bigr)} \bigr)}\) and lemma \ref{local-actions} applied to \(X \rtimes A\) and \(A\) we may assume that \(g \in D_\alpha(\Phi, X_{\mathfrak p})\), where \(D_\alpha(\Phi, (X \rtimes A)_{\mathfrak p}) = D_\alpha(\Phi, X_{\mathfrak p}) \rtimes D_\alpha(\Phi, A_{\mathfrak p})\). Using the identities \(\up g{\bigl(\up ax\bigr)} \equiv \up{\delta(\langle g, a \rangle)\, a}{\bigl(\up gx\bigr)} \pmod{N^{(\infty, \mathfrak p)}}\), \(\up{\delta(g)}{\bigl(\up ax\bigr)} = \up{\delta(\langle g, a \rangle)\, a}{\bigl(\up{\delta(g)}x\bigr)}\) and lemma \ref{root-elimination} applied to \(X \rtimes A\) and \(A\), we may also assume that \(x\) has the domain \(\stlin\bigl(\Sigma, X^{(\infty, \mathfrak p)}\bigr)\) for a thick \(\alpha\)-series \(\Sigma\). But now both sides coincide in \(\stlin\bigl(\Sigma, X^{(\infty, \mathfrak p)}\bigr) \leq \Ker\bigl(p_{1*}^{(\infty, \mathfrak p)}\bigr)\), since their images under \(\stmap\) are the same.

In the same way we check that
\[\langle g, \delta(x) \rangle\, x \equiv \up gx \pmod N\]
for \(g \in G(\Phi, X)\) and \(x \in \Ker(p_{1*})\). It is easy to see that both sides are homomorphisms on \(x\) modulo \(N\). By lemma \ref{costalks} is suffices to check that
\[\langle g, \delta(x) \rangle\, x \equiv \up gx \pmod{N^{(\infty, \mathfrak p)}}\]
for a prime ideal \(\mathfrak p \leqt K\), \(g \in G(\Phi, X_{\mathfrak p})\), and a formal variable \(x \in \Ker\bigl(p_{1*}^{(\infty, \mathfrak p)}\bigr)\). Here both sides are also morphisms of pro-groups \(\Ker\bigl(p_{1*}^{(\infty, \mathfrak p)}\bigr) \to \Ker\bigl(p_{1*}^{(\infty, \mathfrak p)}\bigr) / N^{(\infty, \mathfrak p)}\) for any \(g\). Since \(\langle gh, \delta(x) \rangle\, x = \up g{\bigl(\langle h, \delta(x) \rangle\, x\bigr)}\, \up g{x^{-1}}\, \langle g, \delta(x) \rangle\, x\), \(N^{(\infty, \mathfrak p)}\) is \(G(\Phi, X_{\mathfrak p})\)-invariant, and \(\langle \up ug, \delta(x) \rangle\, x = \up u{\bigl(\bigl\langle g, \delta\bigl(\up{u^{-1}}x\bigr) \bigr\rangle\, \up{u^{-1}}x\bigr)}\), we may apply lemma \ref{local-actions} and assume that \(g \in D_\alpha(X_{\mathfrak p})\). Using the identity \(\langle g, \delta(\up ax) \rangle\, \up ax \equiv \up{\delta(\langle g, a \rangle)\, a}{\bigl(\langle g, \delta(x) \rangle\, x \bigr)} \pmod{N^{(\infty, \mathfrak p)}}\), we may also assume that \(x \in \stlin\bigl(\Sigma, X^{(\infty, \mathfrak p)}\bigr)\) for a thick \(\alpha\)-series \(\Sigma\). Now again both sides coincide in \(\stlin\bigl(\Sigma, X^{(\infty, \mathfrak p)}\bigr)\) since their images under \(\stmap\) are the same.

In easily follows that all actions and the map \(\langle {-}, {=} \rangle\) are well-defined for (\ref{*}) and they form a crossed square.

It remains to check the uniqueness. The action of \(G(\Phi, A)\) on \(\stlin(\Phi, A, X)\) is unique modulo the kernel of \(\stmap \colon \stlin(\Phi, A, X) \to G(\Phi, X)\) and \(\stlin(\Phi, A)\) centralizes this kernel by the axioms of crossed squares. Moreover,
\[\stlin(\Phi, A, X) = [\stlin(\Phi, A), \stlin(\Phi, A, X)]\]
inside \(\stlin(\Phi, A, X) \rtimes \stlin(\Phi, A)\) (since \(\stlin(\Phi, A \rtimes X)\) is perfect) and \(\up g{[a, x]} = [\up ga, \up gx]\) for any \(g \in G(\Phi, A)\), \(a \in \stlin(\Phi, A)\), \(x \in \stlin(\Phi, A, X)\). It follows that the action of \(G(\Phi, A)\) on \(\stlin(\Phi, A, X)\) is unique.

Take any \(g \in G(\Phi, X)\), \(a \in \stlin(\Phi, A)\), \(b \in \stlin(\Phi, A)\) and let \(u = \langle g, a \rangle\), \(v = \langle g, b \rangle\). We have
\[\langle g, [a, b] \rangle = u \bigl(\up{\stmap(a)}{v}\bigr) \bigl(\up{\stmap(aba^{-1})}{u^{-1}}\bigr) \bigl(\up{\stmap([a, b])}{v^{-1}}\bigr).\]
Both \(u\) and \(v\) are uniquely determined modulo the kernel of \(\stmap \colon \stlin(\Phi, A, X) \to G(\Phi, X)\). But this implies that \(\langle g, [a, b] \rangle\) is uniquely determined since
\begin{align*}
w \bigl(\up{\stmap(a)}v\bigr) \bigl(\up{\stmap(aba^{-1})}{w^{-1}}\bigr) &= w \bigl(\up{\stmap(a) \delta(g) \stmap(b) \delta(g)^{-1} \stmap(a)^{-1}}{w^{-1}}\bigr) \bigl(\up{\stmap(a)}v\bigr) \\
&= \bigl\langle \stmap(w), \up{\stmap(a) \delta(g)}{b} \bigr\rangle \bigl(\up{\stmap(a)}v\bigr) = \up{\stmap(a)}v, \\
w \bigl(\up{\stmap(ba^{-1})}{u^{-1}}\bigr) \bigl(\up{\stmap(b a^{-1} b^{-1})}{w^{-1}}\bigr) &= w \bigl(\up{\stmap(b) \delta(g) \stmap(a^{-1}) \delta(g)^{-1} \stmap(b^{-1})}{w^{-1}}\bigr) \bigl(\up{\stmap(ba^{-1})}{u^{-1}}\bigr) \\
&= \bigl\langle \stmap(w), \up{\stmap(b) \delta(g)}{a^{-1}} \bigr\rangle \bigl(\up{\stmap(ba^{-1})}{u^{-1}}\bigr) = \up{\stmap(ba^{-1})}{u^{-1}}
\end{align*}
for any \(w \in \stlin(\Phi, A, X)\) in the kernel of \(\stmap\). The group \(\stlin(\Phi, A)\) is perfect, so the crossed pairing is unique.
\end{proof}

\section{Cosheaf properties}

By theorems \ref{st-pro-group} and \ref{st-x-square}, there is a crossed square
\[\xymatrix@R=30pt@C=72pt@!0{
\stlin\bigl(\Phi, A^{(\infty, k)}\bigr) \ar[r]^\delta \ar[d]_{\stmap} & \stlin(\Phi, A) \ar[d]^\stmap \\
G\bigl(\Phi, A^{(\infty, k)}\bigr) \ar[r]^\delta & G(\Phi, A)
}\]
natural on \(k \in K\) and \(A\). In particular, \(\stlin(\Phi, A^{(\infty, k)})\) are crossed modules over \(G(\Phi, A)\) in a canonical way.

For any morphism \(k \to k'\) in \(\mathcal K\) we denote the morphisms \(A^{(\infty, k')} \to A^{(\infty, k)}\), \(P_\alpha\bigl(A^{(\infty, k')}\bigr) \to P_\alpha\bigl(A^{(\infty, k)}\bigr)\), \(G\bigl(\Phi, A^{(\infty, k')}\bigr) \to G\bigl(\Phi, A^{(\infty, k)}\bigr)\), and \(\stlin\bigl(\Phi, A^{(\infty, k')}\bigr) \to \stlin\bigl(\Phi, A^{(\infty, k)}\bigr)\) by \(\can^{k'}_k\). The next result shows that \(P_\alpha\bigl(A^{(\infty, k)}\bigr)\) are cosheaves of abelian pro-groups unless \(2 \alpha \in \Phi\).

\begin{lemma} \label{param-cosheaf}
Let \(k_i\), \(t_i\), and \(s\) be elements of \(K\) for \(1 \leq i \leq n\) such that \(s \in \sum_{i = 1}^n k_i K\). The pro-group \(P_\alpha\bigl(A^{(\infty, s)}\bigr)\) is generated by \(\can^{sk_i}_s\) with the relations
\begin{itemize}
\item \(\bigl[\can^{sk_i}_s(a), \can^{sk_j}_s(b)\bigr]^\cdot = 0\) for all \(i\) and \(j\);
\item \(\can^{sk_i}_s(a + b) = \can^{sk_i}_s(a) \dotplus \can^{sk_i}_s(b)\);
\item \(\can^{sk_i}_s\bigl(\can^{sk_i k_j}_{sk_i}(a)\bigr) = \can^{sk_j}_s\bigl(\can^{sk_i k_j}_{sk_j}(a)\bigr)\) for \(i \neq j\).
\end{itemize}
unless either \(2 \alpha \in \Phi\) or \(\frac \alpha 2 \in \Phi\). In the case \(2 \alpha \in \Phi\) the pro-group \(P_\alpha\bigl(A^{(\infty, s)}\bigr)\) is generated by \(\mathrm{can}^{sk_i}_s\) with the relations
\begin{itemize}
\item \(\bigl[\can^{sk_i}_s(u), \can^{sk_j}_s(v)\bigr]^\cdot = \can^{sk_i}_s\bigl(\phi\bigl(\inv{\can^{sk_j}_1(\pi(v))}\, \pi(u)\bigr)\bigr)\) for all \(i\) and \(j\);
\item \(\can^{sk_i}_s(u \dotplus v) = \can^{sk_i}_1(u) \dotplus \can^{sk_i}_s(v)\);
\item \(\can^{sk_i}_s\bigl(\can^{sk_i k_j}_{sk_i}(u)\bigr) = \can^{sk_j}_s\bigl(\can^{sk_i k_j}_{sk_j}(u)\bigr)\) for \(i \neq j\).
\end{itemize}
\end{lemma}
\begin{proof}
Recall that a formal projective limit of maps \(u_i \colon X_i \to Y_i\) is an isomorphism in \(\Pro(\Set)\) if and only if for any \(i\) there are an index \(i'\), a morphism \(m \colon i' \to i\), and a map \(v_i \colon Y_{i'} \to X_i\) such that \(u_i \circ v_i = Y_m \colon Y_{i'} \to Y_i\) and \(v_i \circ u_{i'} = X_m \colon X_{i'} \to X_i\) (so \(v_i\) are components of the inverse morphism in \(\Pro(\Set)\)).

Now fix a root \(\alpha\). Let \(G = \varprojlim_{m \in \mathbb N} G_m\) be the pro-group with the presentation from the statement constructed in lemma \ref{universal-pro-group} as a formal projective limit of groups with given presentations. There is an obvious morphism \(G \to P_\alpha\bigl(A^{(\infty, s)}\bigr)\), its inverse consists of the group homomorphisms
\[a^{(s^{m' + m})} \mapsto \sum_{i = 1}^n (at_{im})^{(s^m k_i^m)}\]
unless \(2 \alpha \in \Phi\) or \(\frac \alpha 2 \in \Phi\) and of
\begin{align*}
u^{(s^{m' + m})} &\mapsto \sum_{1 \leq i \leq n}^\cdot (u \cdot t_{im})^{(s^m k_i^m)} \dotplus \sum_{1 \leq i < j \leq n}^\cdot \phi\bigl( (\rho(u) s^m t_{im} t_{jm} k_j^m)^{(s^m k_i^m)} \bigr), \\
\phi\bigl(a^{(s^{m' + m})}\bigr) &\mapsto \sum_{1 \leq i \leq n}^\cdot \phi\bigl((a t_{im})^{(s^m k_i^m)}\bigr)
\end{align*}
for \(2 \alpha \in \Phi\), where \(m' = \max(0, (m - 1) n + 1)\) and \(s^{m'} = \sum_{i = 1}^n k_i^m t_{im}\) for some \(t_{im} \in K\).
\end{proof}

In particular, \(K^{(\infty, k)}\) form a cosheaf of abelian pro-groups. If \(K\) is an integral domain, then any \(K^{(\infty, k)}\) is a sub-pro-ring of \(K\) (since we may identify each \(K^{(s)}\) with \(Ks\)) and \(K / K^{(\infty, k)} = \varprojlim^{\Pro}_n K / K k^n\) is the formal pro-completion at \(k\). For example, in the case \(K = \mathbb C[x, y]\) the continuation of this cosheaf at \(\Spec(K_x) \cup \Spec(K_y)\) is the sub-pro-group \(K^{(\infty, x)} + K^{(\infty, y)} \leq K\) and \(K / \bigl(K^{(\infty, x)} + K^{(\infty, y)}\bigr) = \varprojlim^{\Pro}_n K / (K x^n + K y^n)\) is the pro-group of formal power series in two variables. Recall that the value of the sheaf of localizations (i.e. the structure sheaf of the scheme \(\Spec(K)\)) on this subset coincides with \(K\).

It is easy to see that a formal projective limit of maps \(u_i \colon X_i \to Y_i\) is an epimorphism in \(\Pro(\Set)\) if for any \(i\) there are an index \(i'\) and a morphism \(m \colon i' \to i\) such that the image of \(Y_m \colon Y_{i'} \to Y_i\) is contained in the image of \(u_i\) (this is actually the description of \textit{regular epimorphisms}).

\begin{theorem} \label{st-pro-gluing}
Let \(A\) be an absolute object in \(\mathcal A\) and \(k_i\), \(s\) be elements of \(K\) for \(1 \leq i \leq n\) such that \(s \in \sum_{i = 1}^n k_i K\). Then the pro-group \(\stlin\bigl(\Phi, A^{(\infty, s)}\bigr)\) is generated by \(\can^{sk_i}_s \colon \stlin\bigl(\Phi, A^{(\infty, sk_i)}\bigr) \to \stlin\bigl(\Phi, A^{(\infty, s)}\bigr)\) with the relations
\begin{itemize}
\item \(\can^{sk_i}_s(gh) = \can^{sk_i}_s(g)\, \can^{sk_i}_s(h)\);
\item \(\can^{sk_i}_s(\can^{sk_i k_j}_{sk_i}(g)) = \can^{sk_j}_s(\can^{sk_i k_j}_{sk_j}(g))\) for \(i \neq j\);
\item \(\up{\can^{sk_i}_s(g)}{\can^{sk_j}_s(h)} = \can^{sk_j}_s\bigl(\up{\can^{sk_i}_1(g)}{h}\bigr)\) for \(i \neq j\);
\end{itemize}
where in the last relation the group \(\stlin(\Phi, A)\) canonically strongly acts on \(\stlin\bigl(\Phi, A^{(\infty, sk_j)}\bigr)\).
\end{theorem}
\begin{proof}

By lemma \ref{param-cosheaf}, \(\stlin\bigl(\Phi, A^{(\infty, s)}\bigr)\) is generated by \(x^i_\alpha = x_\alpha \circ \can^{sk_i}_s\) for \(1 \leq i \leq n\) and \(\alpha \in \Phi\) with the relations
\begin{itemize}
\item \(x^i_\alpha(a \dotplus b) = x^i_\alpha(a)\, x^i_\alpha(b)\);
\item \(x^i_\alpha(a) = x^i_{2 \alpha}(a)\) for \(a \in P_{2 \alpha}\bigl(A^{(\infty, sk_i)}\bigr)\);
\item \(x^i_\alpha\bigl(\can^{sk_i k_j}_{sk_i}(a)\bigr) = x^j_\alpha\bigl(\can^{sk_i k_j}_{sk_j}(a)\bigr)\);
\item \([x^i_\alpha(a), x^j_\beta(b)] = \prod_{\substack{i \alpha + j \beta \in \Phi\\ i, j > 0}} x^i_{i \alpha + j \beta}\bigl(f_{\alpha \beta ij}\bigl(a, \can^{sk_j}_1(b)\bigr)\bigr)\) unless \(\alpha\) and \(\beta\) are anti-parallel.
\end{itemize}
Indeed, the Chevalley commutator formula from the standard presentation of \(\stlin\bigl(\Phi, A^{(\infty, s)}\bigr)\) easily reduces to these relations since \(\prod_{i = 1}^n P_\alpha\bigl(A^{(\infty, sk_i)}\bigr) \to P_\alpha\bigl(A^{(\infty, s)}\bigr), (a_i)_i \mapsto \sum^\cdot_i \can^{sk_i}_s(a_i)\) is an epimorphism of pro-sets and the group-theoretic identity
\begin{align*}
\bigl[\prod_{i = 1}^n g_i, \prod_{j = 1}^m h_j\bigr] &= \up{g_1 \cdots g_{n - 1}}{\bigl([g_n, h_1]\, \up{h_1}{[g_n, h_2]} \cdots \up{h_1 \cdots h_{m - 1}}{[g_n, h_m]}\bigr)} \\
&\quad \up{g_1 \cdots g_{n - 2}}{\bigl([g_{n - 1}, h_1]\, \up{h_1}{[g_{n - 1}, h_2]} \cdots \up{h_1 \cdots h_{m - 1}}{[g_{n - 1}, h_m]}\bigr)} \\
&\quad \ldots \\
&\quad \bigl([g_1, h_1]\, \up{h_1}{[g_1, h_2]} \cdots \up{h_1 \cdots h_{m - 1}}{[g_1, h_m]}\bigr)
\end{align*}
holds. On the other hand, this presentation is equivalent to the presentation from the statement since the relations from the statement hold for Steinberg pro-groups and imply the relations from the proof.
\end{proof}

\begin{theorem} \label{st-pro-cosheaf}
Let \(A\) be an absolute object in \(\mathcal A\). The precosheaf \(k \mapsto \stlin\bigl(\Phi, A^{(\infty, k)}\bigr)\) is a cosheaf of crossed pro-modules over \(G(\Phi, A)\) or \(\stlin(\Phi, A)\). In other words, for all \(k_i, s \in K\) for \(1 \leq i \leq n\) with the property \(s \in \sum_{i = 1}^n k_i K\) the crossed pro-module \(\stlin\bigl(\Phi, A^{(\infty, s)}\bigr)\) is the universal crossed pro-module with morphisms \(\can^{sk_i}_s \colon \stlin\bigl(\Phi, A^{(\infty, sk_i)}\bigr) \to \stlin(\bigl(\Phi, A^{(\infty, s)}\bigr)\) such that \(\can^{sk_i}_s \circ \can^{sk_i k_j}_{sk_i} = \can^{sk_j}_s \circ \can^{sk_i k_j}_{sk_j}\).
\end{theorem}
\begin{proof}
Let \(X\) be a crossed pro-module over \(G(\Phi, A)\) or \(\stlin(\Phi, A)\) and \(f_i \colon \stlin\bigl(\Phi, A^{(\infty, sk_i)}\bigr) \to X\) be morphisms of crossed pro-modules such that \(f_i \circ \can^{sk_i k_j}_{sk_i} = f_j \circ \can^{sk_i k_j}_{sk_j}\). By theorem \ref{st-pro-gluing} it follows that there is a unique morphism of pro-groups \(g \colon \stlin\bigl(\Phi, A^{(\infty, s)}\bigr) \to X\) such that \(g \circ \can^{sk_i}_s = f_i\). It remains to check that \(g\) is \(G(\Phi, A)\)-equivariant (or \(\stlin(\Phi, A)\)-equivariant). But the product map \(G(\Phi, A) \times \prod_{i = 1}^n \stlin\bigl(\Phi, A^{(\infty, sk_i)}\bigr) \to G(\Phi, A) \times \stlin\bigl(\Phi, A^{(\infty, s)}\bigr)\) is a epimorphism of pro-sets (this easily follows from theorem \ref{st-pro-group}), so the required equivariance follows from the equivariance of \(f_i\).
\end{proof}

Finally, we generalize lemma \ref{local-actions} to arbitrary multiplicative subsets.

\begin{theorem} \label{weak-local-actions}
There is a unique weak action of all \(G(\Phi, S^{-1} A)\) on \(\stlin\bigl(\Phi, A^{(\infty, S)}\bigr)\) for multiplicative \(S \subseteq K\) and absolute objects \(A\) in \(\mathcal A\) such that the morphisms \(\stlin\bigl(\Phi, A^{(\infty, \mathfrak p)}\bigr) \to \stlin\bigl(\Phi, A^{(\infty, S)}\bigr)\) are \(G(\Phi, S^{-1} A)\)-equivariant for all prime ideals \(\mathfrak p \leqt K\) disjoint with \(S\). Such weak actions are natural on \(A\) and extranatural on \(S\) in the following sense: for any multiplicative subsets \(S \subseteq S'\) the morphism \(\stlin\bigl(\Phi, A^{(\infty, S')}\bigr) \to \stlin\bigl(\Phi, A^{(\infty, S)}\bigr)\) is \(G(\Phi, S^{-1} A)\)-equivariant.
\end{theorem}
\begin{proof}
Uniqueness, naturality on \(A\), and extranaturality on \(S\) follow from lemma \ref{costalks}. In order to show the existence it suffices to consider only multiplicative subset of type \(S = \{1, s, s^2, \ldots\}\) for some \(s \in K\). Let \(g \in G(\Phi, A_s)\). By lemma \ref{local-actions} there is a Zariski covering \(\Spec(A_s) = \cup_{i = 1}^n \Spec(A_{sk_i})\) (we may assume that \(s \in \sum_i k_i K\)) such that the image of \(g\) in every \(G(\Phi, A_{sk_i})\) lies in the product of various \(D_\alpha(\Phi, A_{sk_i})\), so by lemma \ref{root-elimination} \(g\) induces automorphisms of \(\stlin\bigl(\Phi, A^{(\infty, sk_i)}\bigr)\) with the required property. In order to obtain an induced endomorphism of \(\stlin\bigl(\Phi, A^{(\infty, s)}\bigr)\) we use theorem \ref{st-pro-gluing} and the identities
\begin{align*}
\can^{sk_i}_s\bigl(\up g{(hh')}\bigr) &= \can^{sk_i}_s\bigl(\up gh\, \up g{h'}\bigr) = \can^{sk_i}_s\bigl(\up gh)\, \can^{sk_i}_s\bigl(\up g{h'}\bigr); \\
\can^{sk_i}_s\bigl(\up g{\can^{sk_i k_j}_{sk_i}(h)}\bigr) &= \can^{sk_i}_s\bigl(\can^{sk_i k_j}_{sk_i}\bigl(\up gh\bigr)\bigr) = \can^{sk_j}_s\bigl(\can^{sk_i k_j}_{sk_j}\bigl(\up gh\bigr)\bigr) = \can^{sk_j}_s\bigl(\up g{\can^{sk_i k_j}_{sk_j}(h)}\bigr); \\
\up{\can^{sk_i}_s(\up gh)}{\can^{sk_j}_s\bigl(\up g{h'}\bigr)} &= \can^{sk_j}_s\bigl(\up{g\, \can^{sk_i}_1(\stmap(h))}{h'}\bigr).
\end{align*}
It follows that \(g\) induces an endomorphism of \(\stlin\bigl(\Phi, A^{(\infty, s)}\bigr)\) with the required property. The uniqueness of such endomorphisms implies that we have a monoid homomorphism \(G(\Phi, A_s) \to \mathrm{End}\bigl(\stlin\bigl(\Phi, A^{(\infty, s)}\bigr)\bigr)\) with the required property, it necessarily takes images in \(\Aut\bigl(\stlin\bigl(\Phi, A^{(\infty, s)}\bigr)\bigr)\).
\end{proof}

\bibliographystyle{plain}  
\bibliography{references}

\begin{thebibliography}{10}

\bibitem{central-ko2}
S.~B{\"{o}}ge.
\newblock Steinberggruppen von orthogonalen gruppen.
\newblock {\em J. Reine Angew. Math.}, 494:219--236, 1998.

\bibitem{alg-coh}
A.~S. Cigoli, J.~R.~A. Gray, and T.~Van~der Linden.
\newblock Algebraically coherent categories.
\newblock {\em Theory and applications of categories}, 30(54):1864--1905, 2015.

\bibitem{yoga}
R.~Hazrat, A.~Stepanov, N.~Vavilov, and Z.~Zhang.
\newblock The yoga of commutators.
\newblock {\em J. Math. Sci.}, 179:662--678, 2011.

\bibitem{pro-semiabelian}
P.-A. Jacqmin and Z.~Janelidze.
\newblock On stability of exactness properties under the pro-completion.
\newblock Preprint, arXiv:2002.02204, 2020.

\bibitem{rel-k2-keune}
F.~Keune.
\newblock The relativization of {$\mathrm K_2$}.
\newblock {\em J. Algebra}, 54:159--177, 1978.

\bibitem{rel-st-symp}
A.~Lavrenov.
\newblock Relative symplectic {S}teinberg group, 2014.
\newblock Preprint, arXiv:1412.2421.

\bibitem{central-k2c}
A.~Lavrenov.
\newblock Another presentation for symplectic {S}teinberg groups.
\newblock {\em J. Pure Appl. Alg.}, 219(9):3755--3780, 2015.

\bibitem{central-k2d}
A.~Lavrenov and S.~Sinchuk.
\newblock On centrality of even orthogonal $\mathrm{K}_2$.
\newblock {\em J. Pure Appl. Alg.}, 221(5):1134--1145, 2017.

\bibitem{central-k2f}
A.~Lavrenov, S.~Sinchuk, and E.~Voronetsky.
\newblock Centrality of $\mathrm{K}_2$ for {C}hevalley groups: a pro-group
  approach, 2020.
\newblock Preprint, arXiv:2009.03999.

\bibitem{rel-k2-loday}
J.-L. Loday.
\newblock Cohomologie et groupe de {S}teinberg relatifs.
\newblock {\em J. Algebra}, 54:178--202, 1978.

\bibitem{st-jordan}
O.~Loos and E.~Neher.
\newblock {\em Steinberg groups for {J}ordan pairs}.
\newblock Birkh{\"a}user, New York, 2019.

\bibitem{central-k2e}
S.~Sinchuk.
\newblock On centrality of $\mathrm{K}_2$ for {C}hevalley groups of type
  $\mathrm{E}_l$.
\newblock {\em J. Pure Appl. Alg.}, 220(2):857--875, 2016.

\bibitem{central-k2-local}
A.~Stavrova.
\newblock On the congruence kernel of isotropic groups over rings.
\newblock {\em Trans. Amer. Math. Soc.}
\newblock Published electronically.

\bibitem{central-k2-tulen}
M.~S. Tulenbaev.
\newblock Schur multiplier of the group of elementary matrices of finite order.
\newblock {\em J. Sov. Math.}, 17(4):2062--2067, 1981.

\bibitem{central-k2-vdk}
W.~van~der Kallen.
\newblock Another presentation for {S}teinberg groups.
\newblock {\em Indag. Math.}, 80(4):304--312, 1977.

\bibitem{central-ku2}
E.~Voronetsky.
\newblock Centrality of odd unitary $\mathrm{K}_2$-functor, 2020.
\newblock Preprint, arXiv:2005.02926.

\bibitem{central-k2}
E.~Voronetsky.
\newblock Centrality of {$\mathrm K_2$}-functor revisited.
\newblock {\em J. Pure Appl. Alg.}, 225(4), 2021.

\bibitem{rel-st}
E.~Voronetsky.
\newblock Explicit presentation of relative {S}teinberg groups, 2021.
\newblock Preprint, arXiv:2104.09602.

\bibitem{pro-actions}
E.~Voronetsky.
\newblock Actions of pro-groups and pro-rings, 2022.
\newblock Preprint, arXiv:2205.13336.

\bibitem{thesis}
E.~Voronetsky.
\newblock {\em Lower $\mathrm{K}$-theory of odd unitary groups}.
\newblock PhD thesis, St Petersburg University, Saint Petersburg, Russia, 2022.

\bibitem{rel-stu}
E.~Voronetsky.
\newblock A presentation of relative unitary {S}teinberg groups, 2022.
\newblock Preprint, arXiv:2206.11885.

\bibitem{rank-2-k2}
M.~Wendt.
\newblock On homotopy invariance for homology of rank two groups.
\newblock {\em J. Pure Appl. Alg.}, 216(10):2291--2301, 2012.

\end{thebibliography}

\end{document}